\documentclass[12pt,amscd]{amsart}
\usepackage[all]{xy}
\usepackage{graphicx}
\usepackage{amsmath,amsxtra,amssymb,latexsym, amscd,amsthm}

\usepackage{tikz}
\usepackage{ytableau}

 \usepackage{indentfirst}
\usepackage[mathscr]{eucal}
  \usepackage[pagebackref=true]{hyperref}

\newtheorem{thm}{Theorem}[section]

\newtheorem{lem}[thm]{Lemma}
\newtheorem{prop}[thm]{Proposition}
\theoremstyle{definition}
\newtheorem{defn}[thm]{Definition}
\newtheorem{exam}[thm]{Example}

\newtheorem{rem}[thm]{Remark}

\numberwithin{equation}{section}

\DeclareMathOperator{\NN}{\mathbb {N}}

\DeclareMathOperator{\depth}{depth}
\DeclareMathOperator{\pd}{pd}

\DeclareMathOperator{\reg}{reg}

\def\Adm{\operatorname{Adm}}

\def\a {\mathbf a}
\def\b {\mathbf b}

\def\m {\mathfrak m}

\def\k {\mathrm{k}}

\begin{document}

\title{Projective dimension of powers of cover ideal of Ferrers graphs}

\author{Do Trong Hoang}
\address{Faculty of Mathematics and Informatics, Hanoi University of Science and Technology, 1 Dai Co Viet, Bach Mai, Hanoi, Vietnam.}

\email{hoang.dotrong@hust.edu.vn}

\author{Thanh Vu}
\address{Institute of Mathematics, VAST, 18 Hoang Quoc Viet, Hanoi, Vietnam}
\email{vuqthanh@gmail.com}


\date{}

\dedicatory{Dedicated to Professor Nguyen Tu Cuong on the occasion of his 75th birthday}
\commby{}

\begin{abstract}
   Let $\lambda = (\lambda_1, \ldots, \lambda_n)$ be a partition with $\lambda_1 = m$. Denote by $J_\lambda$ the cover ideal in the polynomial ring
\(
S = k[x_1, \ldots, x_n, y_1, \ldots, y_m]
\)
associated to the Ferrers graph corresponding to $\lambda$. Let $d(\lambda)$ denote the number of distinct parts of $\lambda$. We prove that
\[
\pd(S/J_\lambda^t)
=
\min\{t,\; d(\lambda)\} + 1
\]
for all $t \ge 1$.
\end{abstract}

\maketitle

\section{Introduction}
\label{sect_intro}
Let $\lambda = (\lambda_1, \ldots, \lambda_n)$ be a partition, namely, a non-increasing sequence of positive integers. Corso and Nagel \cite{CN} studied the Ferrers ideal associated to the partition $\lambda$ and exhibited various connections between the combinatorics of Young diagrams and the algebraic invariants of the monomial ideals and monomial algebras associated to $I_\lambda$. Hoang and Vu \cite{HV1,HV2} further explored the relationships between Young tableaux and skew tableaux and the algebraic properties and algebraic invariants of the tableau and skew tableau ideals, which generalize Ferrers ideals. 

In this work, we establish another connection between the combinatorics of Young diagrams and that of the cover ideals of the Ferrers graphs associated with $\lambda$. We note that the combinatorics of Young diagrams is itself a fascinating subject of study, arising naturally in the representation theory of symmetric groups and general linear groups \cite{Ful}. We now introduce some notation.

Assume that $\lambda_1 = m$. Denote by \(S = \k[x_1,\ldots,x_n,y_1,\ldots,y_m]
\) the standard graded polynomial ring over a field $\k$. The Ferrers graph associated with a partition $\lambda$, denoted by $G_\lambda$, is the bipartite graph with bipartition
\[
V(G_\lambda) = \{x_1,\ldots,x_n\} \cup \{y_1,\ldots,y_m\}
\]
and edge set
\[
E(G_\lambda)=\bigl\{\{x_i,y_j\} \mid 1 \le j \le \lambda_i \bigr\}.
\]
The edge ideal and cover ideal associated with the Ferrers graph $G_\lambda$ are denoted by $I_\lambda$ and $J_\lambda$, respectively. From the work of Fr\"oberg \cite{Fr} and Corso and Nagel \cite{CN}, the projective dimension and regularity of $I_\lambda$ are known. Since $J_\lambda$ is the Alexander dual of $I_\lambda$, Terai's theorem \cite{Te} allows us to determine the regularity and projective dimension of $J_\lambda$ as well.

While the projective dimension and regularity of powers of $I_\lambda$ are well understood \cite{HHZ,HV1}, the corresponding results for powers of $J_\lambda$ have remained unknown. In this work, we determine the projective dimension of all powers of $J_\lambda$. Let $d(\lambda)$ denote the number of distinct parts of $\lambda$. Our main result is the following.

\begin{thm}\label{thm_pd_power_J}
Let $\lambda$ be a partition. Then, for all $t \ge 1$,
\[
\pd (S/J_\lambda^t) = 1 + \min \{t, d(\lambda)\}.
\]
\end{thm}

It is worth noting that explicit formulas for the depth of powers of ideals are relatively rare, even in the case of edge ideals of graphs \cite{BC,MTV,HHV}. For cover ideals of graphs, when the symbolic depth function is not constant, the only recent explicit results are those for paths \cite{DHNT} and cycles \cite{DHV}. Moreover, although the depth of symbolic powers of cover ideals of graphs is known to be weakly decreasing \cite{HKTT}, it may exhibit several plateau regions before stabilizing. The behavior established in this work for Ferrers graphs is particularly simple: the depth decreases by one at each step until stabilization. This phenomenon resembles that of powers of edge ideals of Cohen--Macaulay trees \cite{HHV}.

\begin{exam}
Let $\lambda = (8,8,8,3,2,2)$. Then $d(\lambda)=3$. Hence, for all $t \ge 1$,
\[
\pd(S/J_\lambda^t)=1+\min\{t,3\}.
\]
\end{exam}

Recently, Dung, Hang, and Vu \cite{DHV}, building on Hochster's formula \cite{Hoc}, translated the problem of computing the depth of symbolic powers of cover ideals into the problem of determining the regularity of certain admissible subgraphs. We refer the reader to the next section for further details.

In the next section, we recall the necessary notation and results concerning admissible subgraphs. In Section~\ref{sec_app}, we apply these tools to establish our main theorem.

\section{Depth of symbolic powers of cover ideals via admissible subgraphs}
In this section, we first review the definitions and basic properties of depth for monomial ideals. We then recall the notion of admissible subgraphs and explain how it can be used to compute the depth of symbolic powers of cover ideals of hypergraphs. Finally, we recall several results concerning Ferrers ideals.

Throughout the first two subsections, let
\( S = \k[x_1,\ldots, x_n]
\)
be a standard graded polynomial ring over a field $\k$ with homogeneous maximal ideal
\(
\m = (x_1,\ldots, x_n).
\)
We begin by recalling the notion of depth. For a finitely generated graded $S$-module $L$, the depth of $L$ is defined by
\[
\depth(L) = \min\{i \mid H_{\m}^i(L) \ne 0\},
\]
where $H^{i}_{\m}(L)$ denotes the $i$-th local cohomology module of $L$ with respect to $\m$.

\subsection{Graphs and their edge ideals and cover ideals} We recall some basic notions from graph theory; for further details, see~\cite{Di}. Let $G$ be a simple graph with vertex set $V(G) = \{1, \ldots, n\}$ and edge set $E(G)$. Throughout, we assume that $G$ has no isolated vertices.

A graph $H$ is a \emph{subgraph} of $G$ if $V(H) \subseteq V(G)$ and $E(H) \subseteq E(G)$. It is an \emph{induced subgraph} of $G$ if, for every subset $e \subseteq V(H)$, the set $e$ is an edge of $H$ if and only if it is an edge of $G$.

A subset $W \subseteq V(G)$ is called a \emph{vertex cover} of $G$ if $W \cap e \neq \emptyset$ for every edge $e \in E(G)$. It is called a \emph{minimal vertex cover} if no proper subset of $W$ is a vertex cover of $G$.

\begin{defn}
Let $G$ be a graph with vertex set $V(G) = \{1, \ldots, n\}$ and edge set $E(G)$. The \emph{edge ideal} and \emph{cover ideal} of $G$, denoted by $I(G)$ and $J(G)$, respectively, are defined by
\[
I(G) = (x_ix_j \mid \{i,j\} \in E(G))
\quad \text{and} \quad
J(G) = \bigcap_{\{i,j\} \in E(G)} (x_i,x_j).
\]
\end{defn}

In particular, $J(G)$ is the Alexander dual of $I(G)$. It is well known that
\[
J(G) = (x_W \mid W \text{ is a minimal vertex cover of } G),
\]
where $x_W = \prod_{i \in W} x_i$.

The $t$-th symbolic power of $J(G)$ is given by
\[
J(G)^{(t)}
=
\bigcap_{ \{i,j\} \in E(G)}
(x_i, x_j)^t.
\]

\subsection{Admissible subgraphs}

\begin{defn}
Let $G$ be a graph with vertex set $V(G) = [n]$. A subgraph $H$ of $G$ is called a \emph{$t$-admissible subgraph} of $G$ if there exists an exponent vector $\a = (a_1, \ldots, a_n) \in \mathbb{N}^n$ such that
\[
E(H)
=
\bigl\{
\{i, j\} \in E(G)
\;\big|\;
a_i + a_j < t
\bigr\}.
\]
The set of all $t$-admissible subgraphs of $G$ is denoted by $\Adm_t(G)$.
\end{defn}

The following result is \cite[Lemma 2.11]{DHV}. We state it in terms of projective dimension, since this formulation is more convenient for our application. The depth and projective dimension are related by the Auslander--Buchsbaum formula.

\begin{lem}\label{lem_depth_admissible}
Let $G$ be a simple graph. Then
\[
\pd(S/J(G)^{(t)})
=
\max \{\reg(I(H)) \mid H \in \Adm_t(G)\}.
\]
\end{lem}

\subsection{Ordered matchings} 
We recall the notion of ordered matching introduced by Constantinescu and Varbaro \cite{CV}, which plays a crucial role in the study of the depth of symbolic powers of edge ideals. Let $G$ be a simple graph. A subset $M \subseteq E(G)$ is called a \emph{matching} of $G$ if no two edges in $M$ share a common vertex. A subset of vertices $U \subseteq V(G)$ is called \emph{independent} if
\( \{u,v\} \notin E(G)
\) for all $u,v \in U$.

\begin{defn}
Let
\[
M = \{ \{ a_1,b_1\}, \ldots, \{a_r,b_r\}\}
\]
be a matching of $G$. The matching $M$ is called an \emph{ordered matching} of $G$ if 
\begin{enumerate}
    \item $\{a_1, \ldots, a_r\}$ is an independent set;
    \item \( \{a_i,b_j\} \in E(G)\) implies that $i \le j$.
\end{enumerate}
The \emph{ordered matching number} of $G$, denoted by $\nu_o(G)$, is the maximum cardinality of an ordered matching of $G$.
\end{defn}

\subsection{Ferrers graphs}\label{subsection_Ferrers} 
Let $\lambda=(\lambda_1,\lambda_2,\ldots,\lambda_n)$ be a partition with
\(m = \lambda_1 \ge \lambda_2 \ge \cdots \ge \lambda_n \ge 1.
\)
By convention, we set $\lambda_0 = m+1$ and $\lambda_{n+1} = 0$. We denote by $\mu$ the conjugate partition of $\lambda$. The Ferrers graph $G_\lambda$ is the bipartite graph on
\[
X = \{x_1, \ldots, x_n\}
\quad \text{and} \quad
Y = \{y_1,\ldots,y_m\}
\]
with
\[
N(x_i) = \{y_1, \ldots, y_{\lambda_i}\}
\quad \text{and} \quad
N(y_j) = \{x_1, \ldots, x_{\mu_j}\}
\]
for all $i = 1, \ldots, n$ and $j = 1, \ldots, m$. The Ferrers ideal associated with the partition $\lambda$ is the edge ideal of the Ferrers graph in the polynomial ring
\(S = \k[x_1,\ldots,x_n,y_1,\ldots,y_m].\) From \cite[Corollary 2.2 and Corollary 2.6]{CN}, we obtain the following result.

\begin{prop}\label{CN_Ilambda} Let $\lambda = (\lambda_1,\ldots,\lambda_n)$ be a partition. Then
\begin{enumerate}
    \item
    \(
    \pd(S/I_{\lambda})
    =
    \max \{\lambda_j + j - 1\mid 1 \le j \le n\},
    \)
    
    \item
    \(
    \reg(S/I_{\lambda}) = 1.
    \)
\end{enumerate}
\end{prop}

We now compute the ordered matchings of Ferrers graphs.

\begin{lem}\label{lem_ordered_Ferrers}
Let $\lambda$ be a partition. Then
\[
\nu_o(G_\lambda) = d(\lambda).
\]
\end{lem}

\begin{proof}
Let
\( a_1 < a_2 < \cdots < a_r
\) be the row indices such that
\[
\lambda_{a_1} > \lambda_{a_2} > \cdots > \lambda_{a_r}.
\]
Since they correspond to distinct parts of $\lambda$, we have $r \le d(\lambda)$. For simplicity of notation, set \(b_i = \lambda_{a_i}
\) for $i = 1, \ldots, r$. Then
\[
M = \{\{x_{a_1}, y_{b_1}\}, \ldots, \{x_{a_r}, y_{b_r}\}\}
\]
is an ordered matching of $G_\lambda$. Hence,
\[
\nu_o(G_\lambda) \ge d(\lambda).
\]

Conversely, suppose that
\[
M = \{\{x_{a_1}, y_{b_1}\}, \ldots, \{x_{a_r}, y_{b_r}\}\}
\]
is an ordered matching of $G_\lambda$. By definition, we have
\[
\{x_{a_i}, y_{b_j}\} \notin E(G_\lambda)
\]
whenever $i > j$. It follows that
\[
\lambda_{a_i} < b_j \le \lambda_{a_j}.
\]
Hence,
\[
\lambda_{a_1} > \lambda_{a_2} > \cdots > \lambda_{a_r},
\]
which implies that $r \le d(\lambda)$. The proof is complete.
\end{proof}

\section{Proof of the main result}\label{sec_app}
In this section, we analyze the admissible subgraphs of Ferrers graphs and prove our main results. Since $G_\lambda$ is bipartite, a result of Herzog, Hibi, and Trung \cite{HHT} implies that
\(
J_\lambda^t = J_\lambda^{(t)}
\)
for all $t \ge 1$. Hence, the results of the previous section apply.

Since $G_\lambda$ is bipartite, it is convenient to write a certificate of admissibility in terms of a tuple
\(
(\a,\b) \in \NN^n \times \NN^m.
\) A subgraph $H$ of $G_\lambda$ is $t$-admissible with respect to $(\a,\b)$ if
\[
\{x_i,y_j\} \in E(H)
\quad \Longleftrightarrow \quad
\{x_i,y_j\} \in E(G_\lambda)
\text{ and } a_i + b_j < t.
\]

We first consider the problem of projective dimension and begin by proving that the maximum bound is attained.
\begin{lem}\label{lem_max_bound_pd}
Let $\lambda = (\lambda_1,\ldots,\lambda_n)$ be a partition with $d(\lambda)$ distinct parts. For every $1 \le t \le d(\lambda)$, there exists an admissible subgraph $H$ of $G_\lambda$ such that
\[
\reg (I(H)) = 1 + t.
\]
\end{lem}

\begin{proof}
We proceed by induction on $t$. The case $t = 1$ is clear, since $G_\lambda$ is the only admissible subgraph and
\( \reg (I(G_\lambda)) = 2.\)

Now assume that $t \ge 2$. In particular, $d(\lambda) \ge 2$. Let $p$ be the smallest index such that
\( \lambda_p > \lambda_{p+1},
\) and let
\[
\lambda' = (\lambda_{p+1}, \ldots, \lambda_n).
\]
Then
\( d(\lambda') = d(\lambda) - 1.
\) By the induction hypothesis, there exists an admissible subgraph $H'$ of $G_{\lambda'}$ such that
\(
\reg I(H') = t.
\)
Assume that the certificate for $H'$ is given by $\a'$ and $\b'$. Define
\[
a_i = t-1 \quad \text{for all } i = 1, \ldots, p, \text { and }
a_j = a'_j \quad \text{for } j > p,
\]
and 
\[
b_j = 1 + b'_j \quad \text{for } j = 1, \ldots, \lambda_p-1,
\text{ and }
b_j = 0 \quad \text{for } j \ge \lambda_p.
\]

Then the admissible subgraph $H$ corresponding to the certificate $(\a,\b)$ is the disjoint union of $H'$ and a complete bipartite graph. Hence,
\[
\reg(I(H)) = \reg(I(H')) + 1 = t+1.
\]
The proof is complete.
\end{proof}

\begin{proof}[Proof of Theorem \ref{thm_pd_power_J}]
By \cite{HKTT}, we know that $\pd (S/J_\lambda^t)$ is weakly increasing and that
\[
\lim_{t\to \infty} \pd (S/J_\lambda^t)
=
1 + \nu_o(G_\lambda).
\]
By Lemma~\ref{lem_depth_admissible}, Lemma~\ref{lem_max_bound_pd}, and Lemma~\ref{lem_ordered_Ferrers}, it remains to prove that if $H$ is a $t$-admissible subgraph of $G_\lambda$, then
\[
\reg (I(H)) \le t+1.
\]

For each $i = 0, \ldots, t-1$, let
\[
Y_i = \{ j \mid b_j = i\}
\qquad \text{and} \qquad
X_i = \{ \ell \mid a_\ell \le t-1-i\}.
\]
Denote by $H_i$ the induced subgraph of $G_\lambda$ on the vertex set $X_i \cup Y_i$. By definition, $H$ is the union of the graphs $H_0,\ldots,H_{t-1}$. Equivalently,
\[
I(H) = I(H_0) + \cdots + I(H_{t-1}).
\]
Since each $H_i$ is a (possibly empty) Ferrers graph, the conclusion follows from a theorem of Kalai and Meshulam \cite{KM}.
\end{proof}

We conclude with a few remarks.

\begin{rem}
\begin{enumerate}
    \item Dung and Vu showed that one can associate a type sequence to a cochordal graph, generalizing the partition associated with a Ferrers graph. It would be interesting to extend our results to symbolic powers of cover ideals of arbitrary cochordal graphs.

    \item The formula above shows that the depth of powers of cover ideals of $G_\lambda$ is independent of the characteristic of the base field $\k$.

    \item All experimental computations were carried out in Macaulay2 \cite{M2}.
\end{enumerate}
\end{rem}

\end{document}